\documentclass[12pt]{article}
\usepackage{amssymb}
\usepackage{amsthm}
\usepackage{amsmath}
\usepackage{graphicx}
\usepackage{enumerate}
\usepackage{pict2e}

\newtheorem{theorem}{Theorem}[section]

\newtheorem{lemma}[theorem]{Lemma}
\newtheorem{proposition}[theorem]{Proposition}

\newtheorem{remark}[theorem]{Remark}
\newtheorem{example}[theorem]{Example}
\newtheorem{corollary}[theorem]{Corollary}

\title{The logic of quantum mechanics incorporating time dimension}
\author{Ivan~Chajda and Helmut~L\"anger}
\date{}

\begin{document}

\footnotetext{Support of the research of the authors by the Austrian Science Fund (FWF), project I~4579-N, and the Czech Science Foundation (GA\v CR), project 20-09869L, entitled ``The many facets of orthomodularity'', is gratefully acknowledged.}

\maketitle

\begin{abstract}
Similarly as classical propositional calculus is based algebraically on Boolean algebras, the logic of quantum mechanics was based on orthomodular lattices by G.~Birkhoff and J.~von~Neumann \cite{BV} and K.~Husimi \cite H. However, this logic does not incorporate time dimension although it is apparent that the propositions occurring in the logic of quantum mechanics are depending on time. The aim of the present paper is to show that so-called tense operators can be introduced also in such a logic for given time set and given time preference relation. In this case we can introduce these operators in a purely algebraic way. We derive several important properties of such operators, in particular we show that they form dynamic pairs and, altogether, a dynamic algebra. We investigate connections of these operators with logical connectives conjunction and implication derived from Sasaki projections. Then we solve the converse problem, namely to find for given time set and given tense operators a time preference relation in order that the resulting time frame induces the given operators. We show that the given operators can be obtained as restrictions of operators induced by a suitable extended time frame.
\end{abstract}



\section{Introduction}

It is well known that any physical theory determines an event-state system $(\mathcal E,\mathcal S)$ where $\mathcal E$ contains the events that may occur with respect to the given system and $\mathcal S$ contains the states that such a physical system may assume. In quantum physics one usually identifies $\mathcal E$ with the set of projection operators of a Hilbert space $\mathbf H$. This set of operators is in bijective correspondence with the set $\mathcal C(\mathbf H)$ of closed subspaces of $\mathbf H$. The set $\mathcal C(\mathbf H)$ ordered by inclusion forms a complete orthomodular lattice. Such lattices were introduced in 1936 by G.~Birkhoff and J.~von~Neumann \cite{BV} and independently in 1937 by K.~Husimi \cite H as a suitable algebraic tool for investigating the logical structure underlying physical theories that, like mentioned quantum mechanics, do not obey the laws of classical logic. For the theory of orthomodular lattices cf.\ the monographs \cite{Be} and \cite K.

However, the logic based on orthomodular lattices does not incorporate the dimension of time. In order to organize this logic as a so-called {\em tense logic} (or {\em time logic} in another terminology, see e.g.\ \cite{CP} and \cite{RU}) we have to introduce the so-called {\em tense operators} $P$, $F$, $H$ and $G$. Their meaning is as follows:
\begin{quote}
$P$ ... ``It has at some time been the case that'', \\
$F$ ... ``It will at some time be the case that'', \\
$H$ ... ``It has always been the case that'', \\
$G$ ... ``It will always be the case that''.
\end{quote}
As the reader may guess, we need a time scale. For this reason a so-called {\em time frame} is introduced. It is a pair $(T,R)$ consisting of a non-empty set $T$ of {\em time} and a non-empty binary relation $R$ on $T$, the so-called relation of {\em time preference}, i.e.\ for $s,t\in T$ we say that $s\mathrel Rt$ means $s$ is {\em before} $t$ or, equivalently, $t$ is {\em after} $s$. For our purposes we will consider sometimes so-called {\em serial relations} $R$ (see \cite{CP}), i.e.\ binary relations $R$ such that for each $s\in T$ there exist $t,u\in T$ with $t\mathrel Rs$ and $s\mathrel Ru$. Every reflexive binary relation is serial. In physical theories $R$ is usually considered to be an order or a quasiorder.

It is worth noticing that our tense operators are in fact special sorts of modal operators, see e.g.\ \cite E and \cite{P57}. The theory of tense logic has its origin in works by A.~N.~Prior (cf.\ \cite{P57} and \cite{P67}) and in the monographs and chapters \cite{FGV}, \cite{GHR}, \cite G and \cite{HR}. For the classical propositional calculus, these operators were studied in \cite{Bu}, for MV-algebras in \cite{DG}, for intuitionistic logic in \cite E, and for De Morgan algebras in \cite{FP}. We hope that our approach for orthomodular lattices constitute an appropriate new achievement.

\section{Preliminaries}

First we recall several concepts used in lattice theory.

An {\em antitone involution} on a poset $(P,\leq)$ is a mapping $'$ from $P$ to $P$ satisfying the following conditions for all $x,y\in P$:
\begin{enumerate}[(i)]
\item $x\leq y$ implies $y'\leq x'$,
\item $x''=x$.
\end{enumerate}
A {\em complementation} on a bounded poset $(P,\leq,0,1)$ is a mapping $'$ from $P$ to $P$ satisfying $x\vee x'=1$ and $x\wedge x'=0$ for all $x\in P$. An {\em orthomodular lattice} is an algebra $(L,\vee,\wedge,{}',0,1)$ of type $(2,2,1,0,0)$ such that $(L,\vee,\wedge,0,1)$ is a bounded lattice, $'$ is an antitone involution that is a complementation and the {\em orthomodular law} holds:
\[
\text{If }x,y\in L\text{ and }x\leq y\text{ then }y=x\vee(y\wedge x').
\]
Let us remark that according to the De Morgan's laws the orthomodular law is equivalent to the following condition:
\[
\text{If }x,y\in L\text{ and }x\leq y\text{ then }x=y\wedge(x\vee y').
\]
In the following we consider a non-trivial (i.e.\ not one-element) complete (orthomodular) lattice $\mathbf L=(L,\vee,\wedge,0,1)$ ($\mathbf L=(L,\vee,\wedge,{}',0,1)$) and a given time frame $(T,R)$. We can define the tense operators as quantifiers over the time frame as follows:
\begin{align*}
P(q)(s) & :=\bigvee\{q(t)\mid t\mathrel Rs\}, \\
F(q)(s) & :=\bigvee\{q(t)\mid s\mathrel Rt\}, \\
H(q)(s) & :=\bigwedge\{q(t)\mid t\mathrel Rs\}, \\
G(q)(s) & :=\bigwedge\{q(t)\mid s\mathrel Rt\}
\end{align*}
for every $q\in L^T$ and $s\in T$. In such a case we call the {\em tense operators} $P$, $F$, $H$ and $G$ to be {\em derived} from or {\em induced} by the time frame $(T,R)$.

In complete orthomodular lattices there is a close connection between the tense operators $P$ and $H$ and the tense operators $F$ and $G$. Namely, if these tense operators are induced by the time frame $(T,R)$ then due to De Morgan's laws we have $H(q)=P(q')'$ and $G(q)=F(q')'$ for all $q\in L^T$.

\begin{example}\label{ex1}
Consider the orthomodular lattice $\mathbf L$ depicted in Fig.~1:

\vspace*{-4mm}

\begin{center}
\setlength{\unitlength}{7mm}
\begin{picture}(12,8)
\put(6,1){\circle*{.3}}
\put(1,4){\circle*{.3}}
\put(4,3){\circle*{.3}}
\put(4,5){\circle*{.3}}
\put(6,3){\circle*{.3}}
\put(6,5){\circle*{.3}}
\put(6,7){\circle*{.3}}
\put(8,3){\circle*{.3}}
\put(8,5){\circle*{.3}}
\put(11,4){\circle*{.3}}
\put(6,1){\line(-5,3)5}
\put(6,1){\line(-1,1)2}
\put(6,1){\line(0,1)2}
\put(6,1){\line(1,1)2}
\put(6,1){\line(5,3)5}
\put(6,7){\line(-5,-3)5}
\put(6,7){\line(-1,-1)2}
\put(6,7){\line(0,-1)2}
\put(6,7){\line(1,-1)2}
\put(6,7){\line(5,-3)5}
\put(6,3){\line(-1,1)2}
\put(6,3){\line(1,1)2}
\put(6,5){\line(-1,-1)2}
\put(6,5){\line(1,-1)2}
\put(4,3){\line(0,1)2}
\put(8,3){\line(0,1)2}
\put(5.85,.3){$0$}
\put(3.35,2.85){$a$}
\put(6.4,2.85){$b$}
\put(8.4,2.85){$c$}
\put(3.35,4.85){$c'$}
\put(6.4,4.85){$b'$}
\put(8.4,4.85){$a'$}
\put(5.85,7.4){$1$}
\put(.35,3.85){$d$}
\put(11.4,3.85){$d'$}
\put(5.2,-.75){{\rm Fig.~1}}
\end{picture}
\end{center}

\vspace*{4mm}

Put $(T,R):=(\{1,2,3,4,5\},\leq)$ and define time depending propositions $p,q\in L^T$ as follows:
\[
\begin{array}{l|l|l|l|l|l}
t    & 1  & 2  & 3  & 4  & 5 \\
\hline
p(t) & c' & b' & c' & a' & b'
\end{array}
\quad\quad\quad
\begin{array}{l|l|l|l|l|l}
t    & 1 & 2  & 3 & 4 & 5 \\
\hline
q(t) & a & b' & d & a & a'
\end{array}
\]
Then we have
\[
\begin{array}{l|l|l|l|l|l}
t        & 1  & 2 & 3 & 4 & 5 \\
\hline
P(p)(t)  & c' & 1  & 1  & 1  & 1 \\
F(p)(t)  & 1  & 1  & 1  & 1  & b' \\
H(p)(t)  & c' & a  & a  & 0  & 0 \\
G(p)(t)  & 0  & 0  & 0  & c  & b'
\end{array}
\quad\quad\quad
\begin{array}{l|l|l|l|l|l}
t        & 1 & 2  & 3 & 4 & 5 \\
\hline
P(q)(t)  & a & b' & 1 & 1 & 1 \\
F(q)(t)  & 1 & 1  & 1 & 1 & a' \\
H(q)(t)  & a & a  & 0 & 0 & 0 \\
G(q)(t)  & 0 & 0  & 0 & 0 & a'
\end{array}
\]
\end{example}

\section{Dynamic pairs}

At first we prove that for tense operators as defined above the pairs $(P,G)$ and $(F,H)$ form so-called {\em dynamic pairs}, thus $(\mathbf L,P,F,H,G)$ is a so-called {\em dynamic algebra} (see \cite{CP} for details).

\begin{theorem}\label{th1}
Let $(L,\vee,\wedge,0,1)$ be a complete lattice, $(T,R)$ a time frame with serial relation $R$, $P$, $F$, $H$ and $G$ denote the tense operators induced by $(T,R)$ and $p,q\in L^T$. Then the following holds:
\begin{enumerate}[{\rm(i)}]
\item $P(0)=F(0)=H(0)=G(0)=0$ and $P(1)=F(1)=H(1)=G(1)=1$,
\item $p\leq q$ implies $P(p)\leq P(q)$, $F(p)\leq F(q)$, $H(p)\leq H(q)$ and $G(p)\leq G(q)$,
\item $PG(q)\leq q\leq GP(q)$, \\
$FH(q)\leq q\leq HF(q)$.
\end{enumerate}
\end{theorem}

\begin{proof}
Let $s\in T$.
\begin{enumerate}[(i)]
\item Since $R$ is serial we have $P(0)(s)=\bigvee\{0\mid t\mathrel Rs\}=0$ and $P(1)(s)=\bigvee\{1\mid t\mathrel Rs\}=1$. The situation for $F$, $H$ and $G$ is analogous.
\item Assume $p\leq q$. Then
\[
p(t)\leq q(t)\leq\bigvee\{q(u)\mid u\mathrel Rs\}=P(q)(s)
\]
for all $t\in T$ with $t\mathrel Rs$ and hence
\[
P(p)(s)=\bigvee\{p(t)\mid t\mathrel Rs\}\leq P(q)(s).
\]
This shows $P(p)\leq P(q)$. The inequality $F(p)\leq F(q)$ can be shown analogously. Moreover,
\[
H(p)(s)=\bigwedge\{p(u)\mid u\mathrel Rs\}\leq p(t)\leq q(t)
\]
for all $t\in T$ with $t\mathrel Rs$ and hence
\[
H(p)(s)\leq\bigwedge\{q(t)\mid t\mathrel Rs\}=H(q)(s).
\]
This shows $H(p)\leq H(q)$. The inequality $G(p)\leq G(q)$ can be shown analogously.
\item
The following are equivalent:
\begin{align*}
                           PG(q)(s) & \leq q(s), \\
\bigvee\{G(q)(t)\mid t\mathrel Rs\} & \leq q(s), \\
                            G(q)(t) & \leq q(s)\text{ for all }t\in T\text{ with }t\mathrel Rs, \\
 \bigwedge\{q(u)\mid t\mathrel Ru\} & \leq q(s)\text{ for all }t\in T\text{ with }t\mathrel Rs.
\end{align*}
Since the last statement is true, the same holds for the first statement. Analogously, one can prove $FH(q)\leq q$. Now the following are equivalent:
\begin{align*}
q(s) & \leq GP(q)(s), \\
q(s) & \leq\bigwedge\{P(q)(t)\mid s\mathrel Rt\}, \\
q(s) & \leq P(q)(t)\text{ for all }t\in T\text{ with }s\mathrel Rt, \\
q(s) & \leq\bigvee\{q(u)\mid u\mathrel Rt\}\text{ for all }t\in T\text{ with }s\mathrel Rt.
\end{align*}
Since the last statement is true, the same holds for the first statement. Analogously, one can prove $q\leq HF(q)$.
\end{enumerate}
\end{proof}

If the operators $P$, $F$, $H$ and $G$ on the complete lattice $\mathbf L$ satisfy (i), (ii) and (iii) of Theorem~\ref{th1} then the quintuple $(\mathbf L,P,F,H,G)$ will be referred to as a {\em dynamic algebra}.

\begin{example}
For $p$ and $q$ of Example~\ref{ex1} we obtain
\[
\begin{array}{l|l|l|l|l|l}
t        & 1  & 2 & 3 & 4 & 5 \\
\hline
p(t)     & c' & b' & c' & a' & b' \\
PG(p)(t) & 0  & 0  & 0  & c  & b' \\
GP(p)(t) & c' & 1  & 1  & 1  & 1
\end{array}
\quad\quad\quad
\begin{array}{l|l|l|l|l|l}
t        & 1 & 2  & 3 & 4 & 5 \\
\hline
q(t)     & a & b' & d & a & a' \\
PG(q)(t) & 0 & 0  & 0 & 0 & a' \\
GP(q)(t) & a & b' & 1 & 1 & 1
\end{array}
\]
showing that $PG(p)\leq p\leq GP(p)$ and $PG(q)\leq q\leq GP(q)$ and, moreover, that these inequalities are strict.
\end{example}

In the following we establish several properties of tense operators on complete lattices that are in accordance with the general approach presented in \cite{CP} and \cite{RU}.

\begin{theorem}\label{th2}
Let $(L,\vee,\wedge,0,1)$ be a complete lattice, $(T,R)$ a time frame with serial relation $R$, $P$, $F$, $H$ and $G$ denote the tense operators induced by $(T,R)$ and $q\in L^T$. Then the following hold:
\begin{enumerate}[{\rm(i)}]
\item $H(q)\leq P(q)$ and $G(q)\leq F(q)$,
\item if $R$ is reflexive then $H(q)\leq q\leq P(q)$ and $G(q)\leq q\leq F(q)$.
\end{enumerate}
\end{theorem}

\begin{proof}
Let $s\in T$.
\begin{enumerate}[(i)]
\item Since $R$ is serial there exists some $u\in T$ with $u\mathrel Rs$ and we have
\[
H(q)(s)=\bigwedge\{q(t)\mid t\mathrel Rs\}\leq q(u)\leq\bigvee\{q(t)\mid t\mathrel Rs\}=P(q)(s).
\]
The proof for $G$ and $F$ is analogous.
\item We have
\[
H(q)(s)=\bigwedge\{q(t)\mid t\mathrel Rs\}\leq q(s)\leq\bigvee\{q(t)\mid t\mathrel Rs\}=P(q)(s).
\]
The proof for $G$ and $F$ is analogous.
\end{enumerate}
\end{proof}

We define $A\leq B$ for $A,B\in\{P,F,H,G\}$ by $A(q)\leq B(q)$ for all $q\in L^T$.

\begin{theorem}
Let $(L,\vee,\wedge,0,1)$ be a complete lattice, $(T,R)$ a time frame with reflexive $R$, $P$, $F$, $H$ and $G$ denote the tense operators induced by $(T,R)$, $A\in\{P,F,H,G\}$, $B\in\{P,F\}$ and $C\in\{H,G\}$. Then the following hold:
\begin{enumerate}[{\rm(i)}]
\item $A\leq AB$ and $AC\leq A$,
\item if $R$ is, moreover, transitive then $AA=A$.
\end{enumerate}
\end{theorem}

\begin{proof}
\
\begin{enumerate}[(i)]
\item This follows from Theorems~\ref{th2} and \ref{th1}.
\item According to (i) we have $P\leq PP$. Let $q\in L^T$ and $s\in T$. Then the following are equivalent:
\begin{align*}
                           PP(q)(s) & \leq P(q)(s), \\
\bigvee\{P(q)(t)\mid t\mathrel Rs\} & \leq P(q)(s), \\
                            P(q)(t) & \leq P(q)(s)\text{ for all }t\in T\text{ with }t\mathrel Rs, \\
   \bigvee\{q(u)\mid u\mathrel Rt\} & \leq P(q)(s)\text{ for all }t\in T\text{ with }t\mathrel Rs, \\
                               q(u) & \leq \bigvee\{q(w)\mid w\mathrel Rs\}\text{ for all }t\in T\text{ with }t\mathrel Rs\text{ and all }u\in T\text{ with} \\
															      & \hspace*{6mm}u\mathrel Rt.
\end{align*}
Since, due to transitivity of $R$, $t\mathrel Rs$ and $u\mathrel Rt$ together imply $u\mathrel Rs$, the last statement is true and hence the same holds for the first statement. Therefore $PP\leq P$, and together we obtain $PP=P$. The proof of $FF=F$ is analogous. According to (i) we have $HH\leq H$. Now the following are equivalent:
\begin{align*}
                           H(q)(s) & \leq HH(q)(s), \\
                           H(q)(s) & \leq\bigwedge\{H(q)(t)\mid t\mathrel Rs\}, \\
                           H(q)(s) & \leq H(q)(t)\text{ for all }t\in T\text{ with }t\mathrel Rs, \\
                           H(q)(s) & \leq\bigwedge\{q(u)\mid u\mathrel Rt\}\text{ for all }t\in T\text{ with }t\mathrel Rs, \\
\bigwedge\{q(w)\mid w\mathrel Rs\} & \leq q(u)\text{ for all }t\in T\text{ with }t\mathrel Rs\text{ and all }u\in T\text{ with }u\mathrel Rt.
\end{align*}
Since, due to transitivity of $R$, $t\mathrel Rs$ and $u\mathrel Rt$ together imply $u\mathrel Rs$, the last statement is true and hence the same holds for the first statement. This shows $H\leq HH$, and together we obtain $HH=H$. The proof of $GG=G$ is analogous.
\end{enumerate}
\end{proof}

\section{Connections with logical connectives}

Considering the logic based on an orthomodular lattice $(L,\vee,\wedge,{}',0,1)$ one can ask for logical connectives. One way how to introduce the {\em conjunction} $\odot$ and the implication $\rightarrow$ is based on the so-called {\em Sasaki projections} (see \cite{Be}). This method was successfully used by the authors in \cite{CL17a} and \cite{CL17b} for investigating left adjointness. Let us recall the corresponding definitions:
\begin{equation}\label{equ1}
\begin{array}l
x\odot y:=(x\vee y')\wedge y, \\
x\rightarrow y:=(y\wedge x)\vee x'
\end{array}
\end{equation}
for all $x,y\in L$. The Sasaki projection $p_y$ on $[0,y]$ is given by $p_y(x):=(x\vee y')\wedge y$ for all $x\in L$. Hence we have $x\odot y=p_y(x)$ and $x\rightarrow y=\big(p_x(y')\big)'$ for all $x,y\in L$.

The following result was proved in \cite{CL17a} and \cite{CL17b}.

\begin{proposition}\label{prop1}
Let $(L,\vee,\wedge,{}',0,1)$ be an orthomodular lattice, $\odot$ and $\rightarrow$ defined by {\rm(\ref{equ1})} and $a,b,c\in L$. Then the following holds:
\begin{enumerate}[{\rm(i)}]
\item $a\odot1=1\odot a=a$,
\item $a\odot b\leq c$ if and only if $a\leq b\rightarrow c$ {\rm(}left adjointness{\rm)},
\item $a'=a\rightarrow0$.
\end{enumerate}
\end{proposition}

The following lemma will be used in the next proof.

\begin{lemma}\label{lem1}
Let $(L,\vee,\wedge,{}',0,1)$ be an orthomodular lattice, $\odot$ and $\rightarrow$ defined by {\rm(\ref{equ1})} and $a,b\in L$. Then the following holds:
\begin{enumerate}[{\rm(i)}]
\item $(a\rightarrow b)\odot a=a\wedge b$,
\item $a\leq b\rightarrow(a\odot b)$.
\end{enumerate}
\end{lemma}

\begin{proof}
\
\begin{enumerate}[(i)]
\item Using the orthomodular law we obtain
\[
(a\rightarrow b)\odot a=\big((b\wedge a)\vee a'\vee a'\big)\wedge a=a\wedge\big((a\wedge b)\vee a'\big)=a\wedge b.
\]
\item Using again the orthomodular law we obtain
\[
a\leq a\vee b'=b'\vee\big((a\vee b')\wedge b\big)=\big((a\vee b')\wedge b\wedge b\big)\vee b'=b\rightarrow(a\odot b).
\]
\end{enumerate}
\end{proof}

Our next task is to show connections of tense operators with logical connectives $\odot$ and $\rightarrow$. Using Proposition~\ref{prop1} and Lemma~\ref{lem1} we can prove the following theorem.

\begin{theorem}\label{th6}
Let $(L,\vee,\wedge,{}',0,1)$ be a complete orthomodular lattice, $(T,R)$ a time frame, $P$, $F$, $H$ and $G$ denote the tense operators induced by $(T,R)$ and $A\in\{P,F,H,G\}$. Then the following two assertions are equivalent:
\begin{enumerate}[{\rm(i)}]
\item $A(x)\odot A(y)\leq A(x\odot y)$ for all $x,y\in L^T$,
\item $A(x\rightarrow y)\leq A(x)\rightarrow A(y)$ for all $x,y\in L^T$.
\end{enumerate}
\end{theorem}

\begin{proof}
Let $p,q\in L^T$. First assume (i). According to (i) we have
\[
A(p\rightarrow q)\odot A(p)\leq A\big((p\rightarrow q)\odot p\big)
\]
Now
\[
(p\rightarrow q)\odot p=p\wedge q
\]
because of Lemma~\ref{lem1} and hence
\[
A\big((p\rightarrow q)\odot p\big)=A(p\wedge q).
\]
Applying Theorem~\ref{th1} to $p\wedge q\leq q$ yields
\[
A(p\wedge q)\leq A(q).
\]
Altogether, we obtain
\[
A(p\rightarrow q)\odot A(p)\leq A(q).
\]
Thus, by Proposition~\ref{prop1} we conclude
\[
A(p\rightarrow q)\leq A(p)\rightarrow A(q)
\]
showing (ii). Conversely, assume (ii). According to Lemma~\ref{lem1} we have
\[
p\leq q\rightarrow(p\odot q).
\]
Applying Theorem~\ref{th1} we conclude
\[
A(p)\leq A\big(q\rightarrow(p\odot q)\big).
\]
Using (ii) we obtain
\[
A\big(q\rightarrow(p\odot q)\big)\leq A(q)\rightarrow A(p\odot q).
\]
Altogether, we have
\[
A(p)\leq A(q)\rightarrow A(p\odot q).
\]
Thus, by Proposition~\ref{prop1} we conclude
\[
A(p)\odot A(q)\leq A(p\odot q)
\]
showing (i).
\end{proof}

However, we can prove also further interesting connections between these operators.

\begin{theorem}\label{th7}
Let $(L,\vee,\wedge,{}',0,1)$ be a complete orthomodular lattice, $(T,R)$ a time frame with reflexive $R$, $P$, $F$, $H$ and $G$ denote the tense operators induced by $(T,R)$, $A,A_1,A_2\in\{P,F\}$, $B,B_1,B_2\in\{H,G\}$ and $p,q\in L^T$. Then the following holds:
\begin{enumerate}[{\rm(i)}]
\item $p\leq q\rightarrow A_1\big(A_2(p)\odot q\big)$,
\item $B(p\odot q)\leq A(p)\odot q$,
\item $B(p)\leq q\rightarrow A(p\odot q)$,
\item $B_1\big(B_2(p)\odot q\big)\leq p\odot q$,
\item $p\rightarrow q\leq A_1\big(p\rightarrow A_2(q)\big)$,
\item $B(p\rightarrow q)\odot p\leq A(q)$,
\item $p\rightarrow B(q)\leq A(p\rightarrow q)$,
\item $B_1\big(p\rightarrow B_2(q)\big)\odot p\leq q$.
\end{enumerate}
\end{theorem}

\begin{proof}
We use Theorems~\ref{th1} and \ref{th2}, {\rm(\ref{equ1})} and Proposition~\ref{prop1}.
\begin{enumerate}[(i)]
\item We have
\begin{align*}
     p\odot q & \leq A_1(p\odot q)\text{ by Theorem~\ref{th2}}, \\
            p & \leq A_2(p)\text{ by Theorem~\ref{th2}}, \\
     p\odot q & \leq A_2(p)\odot q\text{ by (\ref{equ1})}, \\
A_1(p\odot q) & \leq A_1\big(A_2(p)\odot q\big)\text{ by Theorem~\ref{th1}}, \\
     p\odot q & \leq A_1\big(A_2(p)\odot q\big)\text{ by transitivity}, \\
            p & \leq q\rightarrow A_1\big(A_2(p)\odot q\big)\text{ by Proposition~\ref{prop1}}.
\end{align*}
The other statements follow in an analogous way.
\item follows from $B(p\odot q)\leq B\big(A(p)\odot q\big)\leq A(p)\odot q$,
\item follows from $B(p)\odot q\leq A\big(B(p)\odot q\big)\leq A(p\odot q)$ by applying Proposition~\ref{prop1},
\item follows from $B_1\big(B_2(p)\odot q\big)\leq B_1(p\odot q)\leq p\odot q$,
\item follows from $p\rightarrow q\leq A_1(p\rightarrow q)\leq A_1\big(p\rightarrow A_2(q)\big)$,
\item follows from $B(p\rightarrow q)\leq B\big(p\rightarrow A(q)\big)\leq p\rightarrow A(q)$ by applying Proposition~\ref{prop1},
\item follows from $p\rightarrow B(q)\leq A\big(p\rightarrow B(q)\big)\leq A(p\rightarrow q)$,
\item follows from $B_1\big(p\rightarrow B_2(q)\big)\leq B_1(p\rightarrow q)\leq p\rightarrow q$ by applying Proposition~\ref{prop1}.
\end{enumerate}
\end{proof}

\section{A construction of the time frame}

A tense logic is established if for a given logic a time frame $(T,R)$ exists such that the lattice together with the tense operators forms a dynamic algebra and the logical connectives are related with tense operators in the way shown in Section~4. Hence, if such a logic incorporating time dimension is created, we can define tense operators $P$, $F$, $H$ and $G$. The question is whether, conversely, for given time set $T$ and given tense operators there exists a suitable time frame $(T,R)$ such that the given tense operators are derived from it. In other words, we ask if for given tense operators $P$, $F$, $H$ and $G$ on a time set $T$ one can find some time preference relation $R$ such that these operators are induced by $(T,R)$. To show that this is possible is the goal of Section~5.

If $P$, $F$, $H$ and $G$ are tense operators on a complete lattice $(L,\vee,\wedge,0,1)$ with time set $T$ then the relations
\begin{align*}
R_1 & :=\{(s,t)\in T^2\mid q(s)\leq P(q)(t)\text{ and }q(t)\leq F(q)(s)\text{ for all }q\in L^T\}, \\
R_2 & :=\{(s,t)\in T^2\mid H(q)(t)\leq q(s)\text{ and }G(q)(s)\leq q(t)\text{ for all }q\in L^T\}, \\
R_3 & :=R_1\cap R_2
\end{align*}
are called {\em the relation induced by $P$ and $F$}, {\em the relation induced by $H$ and $G$} and {\em the relation induced by $P$, $F$, $H$ and $G$}, respectively.

Observe that whenever tense operators $P$, $F$, $H$ and $G$ on a complete lattice $\mathbf L$ are induced by an arbitrarily given time frame then $(\mathbf L,P,F,H,G)$ forms a dynamic algebra, and if, moreover, $\mathbf L$ is a complete orthomodular lattice then Theorems~\ref{th6} and \ref{th7} hold for these operators.

At first we show the relationship between given tense operators $P$ and $F$ and the corresponding operators $P^*$ and $F^*$ induced by the time frame $(T,R)$ where $R$ is induced by $P$ and $F$.

\begin{theorem}\label{th8}
Let $P$ and $F$ be tense operators on a complete lattice $(L,\vee,\wedge,0,1)$ with time set $T$, $R$ denote the relation induced by these operators and $P^*$ and $F^*$ denote the tense operators induced by the time frame $(T,R)$. Then $P^*\leq P$ and $F^*\leq F$.
\end{theorem}

\begin{proof}
If $q\in L^T$ and $s\in T$ then
\begin{align*}
P^*(q)(s) & =\bigvee\{q(t)\mid t\mathrel Rs\}\leq P(q)(s), \\
F^*(q)(s) & =\bigvee\{q(t)\mid s\mathrel Rt\}\leq F(q)(s).
\end{align*}
\end{proof}

Analogously, one can prove

\begin{theorem}\label{th9}
Let $H$ and $G$ be tense operators on a complete lattice $(L,\vee,\wedge,0,1)$ with time set $T$, $R$ denote the relation induced by these operators and $H^*$ and $G^*$ denote the tense operators induced by the time frame $(T,R)$. Then $H\leq H^*$ and $G\leq G^*$.
\end{theorem}

From Theorems~\ref{th8} and \ref{th9} we obtain

\begin{corollary}\label{cor1}
Let $P$, $F$, $H$ and $G$ be tense operators on a complete lattice $(L,\vee,\wedge,0,1)$ with time set $T$, $R$ denote the relation induced by these operators and $P^*$, $F^*$, $H^*$ and $G^*$ denote the tense operators induced by the time frame $(T,R)$. Then $P^*\leq P$, $F^*\leq F$, $H\leq H^*$ and $G\leq G^*$.
\end{corollary}

\begin{example}\label{ex2}
Consider the lattice $\mathbf L$ and the time set $T=\{1,2,3,4,5\}$ from Example~\ref{ex1}. Define new tense operators $P$, $F$, $H$ and $G$ as follows:
\[
P(q)(t):=\left\{
\begin{array}{ll}
q(t) & \text{if }t=2, \\
1    & \text{otherwise}
\end{array}
\right.\quad\quad\quad F(q)(t):=\left\{
\begin{array}{ll}
q(t) & \text{if }t=1, \\
1    & \text{otherwise}
\end{array}
\right.
\]
\[
H(q)(t):=\left\{
\begin{array}{ll}
q(t) & \text{if }t=1, \\
0    & \text{otherwise}
\end{array}
\right.\quad\quad\quad G(q)(t):=\left\{
\begin{array}{ll}
q(t) & \text{if }t=2, \\
0    & \text{otherwise}
\end{array}
\right.
\]
for all $q\in L^T$ and all $t\in T$. Note that these operators satisfy the conditions
\[
H(q)\leq q\leq P(q)\text{ and }G(q)\leq q\leq F(q)
\]
for all $q\in L^T$ which were considered in Theorem~\ref{th2}. Let $R$ denote the relation induced by $P$, $F$, $H$ and $G$. Then $R=\{1\}^2\cup\{2\}^2\cup\{3,4,5\}^2$. This can be seen as follows: Obviously, $\{1\}^2\cup\{2\}^2\cup\{3,4,5\}^2\subseteq R$. Now let $(s,t)\in R$. \\
$s=1\neq t$ would imply $q(t)\leq F(q)(1)=q(1)$ for all $q\in L^T$, a contradiction. \\
$s=2\neq t$ would imply $q(2)=G(q)(2)\leq q(t)$ for all $q\in L^T$, a contradiction. \\
$s\neq1=t$ would imply $q(1)=H(q)(1)\leq q(s)$ for all $q\in L^T$, a contradiction. \\
$s\neq2=t$ would imply $q(s)\leq P(q)(2)=q(2)$ for all $q\in L^T$, a contradiction. \\
This shows $R=\{1\}^2\cup\{2\}^2\cup\{3,4,5\}^2$. For $p$ from Example~\ref{ex1} we have
\[
\begin{array}{l|l|l|l|l|l}
t & 1 & 2 & 3 & 4 & 5 \\
\hline
p(t)      & c' & b' & c' & a' & b' \\
P(p)(t)   & 1  & b' & 1  & 1  & 1 \\
F(p)(t)   & c' & 1  & 1  & 1  & 1 \\
P^*(p)(t) & c' & b' & 1  & 1  & 1 \\
F^*(p)(t) & c' & b' & 1  & 1  & 1
\end{array}
\]
showing $P^*\leq P$ and $F^*\leq F$ in accordance with Corollary~\ref{cor1}, but $P^*\neq P$ and $F^*\neq F$, thus this inequality is strict.
\end{example}

\begin{remark}
Although the new tense operators $P^*$, $F^*$, $H^*$ and $G^*$ constructed as shown in Corollary~\ref{cor1} satisfy only the inequalities $P^*\leq P$, $F^*\leq F$, $H\leq H^*$ and $G\leq G^*$ and, by Example~\ref{ex2}, these inequalities may be strict, it is almost evident from the construction of these operators that $(\mathbf L,P^*,F^*,H^*,G^*)$ forms a dynamic algebra and that these operators are connected with the logical connectives $\odot$ and $\rightarrow$ in the way shown in Theorems~\ref{th6} and \ref{th7} provided $\mathbf L$ is a complete orthomodular lattice.
\end{remark}

Conversely, if a time frame $(T,R)$ on a complete lattice is given and we consider the tense operators $P$, $F$, $H$ and $G$ induced by $(T,R)$ then the relation induced by these operators coincides with $R$, see the following result.

\begin{theorem}\label{th4}
Let $(L,\vee,\wedge,0,1)$ be a complete lattice, $(T,R)$ a time frame, $P$, $F$, $H$ and $G$ denote the tense operators induced by $(T,R)$ and $R^*$ denote the relation induced by these operators. Then $R=R^*$ and hence the tense operators induced by the time frame $(T,R^*)$ coincide with those induced by the time frame $(T,R)$.
\end{theorem}

\begin{proof}
If $s\mathrel Rt$ then
\begin{align*}
H(q)(t) & =\bigwedge\{q(u)\mid u\mathrel Rt\}\leq q(s)\leq\bigvee\{q(u)\mid u\mathrel Rt\}=P(q)(t), \\
G(q)(s) & =\bigwedge\{q(u)\mid s\mathrel Ru\}\leq q(t)\leq\bigvee\{q(u)\mid s\mathrel Ru\}=F(q)(s)
\end{align*}
for all $q\in L^T$ and hence $s\mathrel{R^*}t$. This shows $R\subseteq R^*$. Now assume $R\neq R^*$. Then there exists some $(s,t)\in R^*\setminus R$. For every $u\in T$ let $q_u$ denote the following element of $L^T$:
\[
q_u(t):=\left\{
\begin{array}{ll}
1 & \text{if }t=u, \\
0 & \text{otherwise}
\end{array}
\right.
\]
($t\in T$). Now we would obtain
\[
1=q_s(s)\leq P(q_s)(t)=\bigvee\{q_s(u)\mid u\mathrel Rt\}=\bigvee\{0\mid u\mathrel Rt\}=0,
\]
a contradiction. This shows $R=R^*$.
\end{proof}

\begin{remark}
Assume that tense operators $P$, $F$, $H$ and $G$ on a complete lattice with time set $T$ are given. We want to know if such operators are induced by a suitable time frame with possibly unknown time preference relation. We construct the relation $R$ on $T$ induced by these operators and then we construct the tense operators $P^*$, $F^*$, $H^*$ and $G^*$ induced by the time frame $(T,R)$. Now two cases can happen: Either $P^*=P$, $F^*=F$, $H^*=H$ and $G^*=G$ or at least one of these equalities is violated, it means it is a proper inequality. In the first case the given operators $P$, $F$, $H$ and $G$ are induced by the time frame $(T,R)$ whereas in the second case $P$, $F$, $H$ and $G$ are not induced by any time frame because of Theorems~\ref{th4}.
\end{remark}

In Theorem~\ref{th4} we showed that if a complete lattice $(L,\vee,\wedge,0,1)$ and a time frame $(T,R)$ are given and $P$, $F$, $H$ and $G$ denote the tense operators induced by this time frame then the relation $R^*$ induced by these operators coincides with $R$. If, conversely, the tense operators $P$ and $F$ are given on a complete lattice with a given time set $T$, we can ask whether we can construct a relation inducing these operators. In Theorem~\ref{th8} we showed that if $R$ is induced by given tense operators $P$ and $F$ on a given time set $T$ in the complete lattice $\mathbf L$ then the operators $P^*$ and $F^*$ induced by the time frame $(T,R)$ need not coincide with $P$ and $F$, respectively, they satisfy only the inequalities $P^*\leq P$ and $F^*\leq F$. However, such tense operators $P^*$ and $F^*$ are still related with the logical connectives $\odot$ and $\rightarrow$ as shown in Theorems~\ref{th6} and \ref{th7} provided the complete lattice $\mathbf L$ is orthomodular. We are going to show that the given time set $T$ can be extended to some set $\bar T$ and $R$ can be extended to some binary relation $\bar R$ on $\bar T$ such that the tense operators induced by the time frame $(\bar T,\bar R)$ can be considered in some sense as extensions of the given tense operators $P$ and $F$, respectively. Put
\begin{equation}\label{equ2}
\bar T:=T_1\cup T\cup T_2\text{ where }T_1:=T\times\{1\}\text{ and }T_2:=T\times\{2\}.
\end{equation}
We extend our ``world'' $L^T$ by adding two of its copies, so-called ``parallel worlds'', namely the ``past'' $L^{T_1}$ and the ``future'' $L^{T_2}$. In this way we obtain our ``new world'' $L^{\bar T}$ over the extended time set $\bar T$. We also extend our time depending propositions $q\in L^T$ to $\bar q\in L^{\bar T}$ by defining
\begin{equation}\label{equ3}
\begin{array}l
\bar q\big((s,1)\big):=P(q)(s), \\
\bar q(s):=q(s), \\
\bar q\big((s,2)\big):=F(q)(s) \\
\text{for all }s\in T.
\end{array}
\end{equation}
Now we show that the given operators $P$ and $F$ can be considered in some sense as restrictions of the operators $\bar P$ and $\bar F$ induced by the time frame $(\bar T,\bar R)$, respectively.

\begin{theorem}\label{th3}
Let $P$ and $F$ be tense operators on a complete lattice $(L,\vee,\wedge,0,1)$ with time set $T$ and $R$ denote the relation induced by these operators. Define $\bar T$ by {\rm(\ref{equ2})}, put
\[
\bar R:=\{\big((s,1),s\big)\mid s\in T\}\cup R\cup\{\big(s,(s,2)\big)\mid s\in T\}.
\]
and let $\bar P$ and $\bar F$ denote the tense operators induced by the time frame $(\bar T,\bar R)$. Moreover, for every $q\in L^T$ let $\bar q\in L^{\bar T}$ denote the extension of $q$ defined by {\rm(\ref{equ3})}. Then $\bar R|T=R$ and
\[
\big(\bar P(\bar q)\big)|T=P(q)\text{ and }\big(\bar F(\bar q)\big)|T=F(q)
\]
for all $q\in L^T$.
\end{theorem}

\begin{proof}
We have $\bar R|T=\bar R\cap T^2=R$. If $q\in L^T$ and $s\in T$ then $q(t)\leq P(q)(s)$ for all $t\in T$ with $t\mathrel Rs$ and hence $\bigvee\{q(t)\mid t\mathrel Rs\}\leq P(q)(s)$ which implies
\begin{align*}
\bar P(\bar q)(s) & =\bigvee\{\bar q(\bar t)\mid\bar t\bar Rs\}=\bar q\big((s,1)\big)\vee\bigvee\{\bar q(t)\mid t\mathrel Rs\}=P(q)(s)\vee\bigvee\{q(t)\mid t\mathrel Rs\}= \\
                  & =P(q)(s)
\end{align*}
showing $\big(\bar P(\bar q)\big)|T=P(q)$. Analogously, one can prove $\big(\bar F(\bar q)\big)|T=F(q)$.
\end{proof}

An analogous result holds for $H$ and $G$ instead of $P$ and $F$, respectively, but the extensions of $q\in L^T$ to $\bar q\in L^{\bar T}$ must be slightly modified.

\begin{theorem}
Let $H$ and $G$ be tense operators on a complete lattice $(L,\vee,\wedge,0,1)$ with time set $T$ and $R$ denote the relation induced by these operators. Define $\bar T$ by {\rm(\ref{equ2})}, put
\[
\bar R:=\{\big((s,1),s\big)\mid s\in T\}\cup R\cup\{\big(s,(s,2)\big)\mid s\in T\}.
\]
and let $\bar H$ and $\bar G$ denote the tense operators induced by the time frame $(\bar T,\bar R)$. Moreover, for every $q\in L^T$ let $\bar q\in L^{\bar T}$ denote the extension of $q$ defined by
\begin{align*}
\bar q\big((s,1)\big) & :=H(q)(s), \\
            \bar q(s) & :=q(s), \\
\bar q\big((s,2)\big) & :=G(q)(s)
\end{align*}
for all $s\in T$. Then $\bar R|T=R$ and
\[
\big(\bar H(\bar q)\big)|T=H(q)\text{ and }\big(\bar G(\bar q)\big)|T=G(q)
\]
for all $q\in L^T$.
\end{theorem}

\begin{example}
Consider the time set $T$, the proposition $p$ and the tense operators $P$ and $F$ from Example~\ref{ex1} and write $ti$ instead of $(t,i)$ for $t\in T$ and $i=1,2$. Let $R$ denote the relation induced by $P$ and $F$. Then
\begin{align*}
     R & =\{(s,t)\in\{1,2,3,4,5\}^2\mid s\leq t\}, \\
\bar R & =\{(s1,s)\mid s\in T\}\cup R\cup\{(s,s2)\mid s\in T\}.
\end{align*}
Let $\bar P$ and $\bar F$ denote the tense operators induced by the time frame $(\bar T,\bar R)$. Then we have
\[
\begin{array}{l|l|l|l|l|l|l|l|l|l|l|l|l|l|l|l}
\bar t                 & 11 & 21 & 31 & 41 & 51 & 1  & 2  & 3  & 4  & 5  & 12 & 22 & 32 & 42 & 52 \\
\hline
p(t)                   &    &    &    &    &    & c' & b' & c' & a' & b' &    &    &    &    & \\
P(p)(t)                &    &    &    &    &    & c' & 1  & 1  & 1  & 1  &    &    &    &    & \\
F(p)(t)                &    &    &    &    &    & 1  & 1  & 1  & 1  & b' &    &    &    &    & \\
\bar p(\bar t)         & c' & 1  & 1  & 1  & 1  & c' & b' & c' & a' & b' & 1  & 1  & 1  & 1  & b' \\
\bar P(\bar p)(\bar t) & 0  & 0  & 0  & 0  & 0  & c' & 1  & 1  & 1  & 1  & c' & b' & c' & a' & b' \\
\bar F(\bar p)(\bar t) & c' & b' & c' & a' & b' & 1  & 1  & 1  & 1  & b' & 0  & 0  & 0  & 0  & 0
\end{array}
\]
where
\[
T_1=\{11,21,31,41,51\}, T=\{1,2,3,4,5\}\text{ and }T_2=\{12,22,32,42,52\}.
\]
Evidently, $\bar R|T=R$, $\big(\bar P(\bar q)\big)|T=P(q)$ and $\big(\bar F(\bar q)\big)|T=F(q)$ in accordance with Theorem~\ref{th3}.
\end{example}

Authors' addresses:

Ivan Chajda \\
Palack\'y University Olomouc \\
Faculty of Science \\
Department of Algebra and Geometry \\
17.\ listopadu 12 \\
771 46 Olomouc \\
Czech Republic \\
ivan.chajda@upol.cz

Helmut L\"anger \\
TU Wien \\
Faculty of Mathematics and Geoinformation \\
Institute of Discrete Mathematics and Geometry \\
Wiedner Hauptstra\ss e 8-10 \\
1040 Vienna \\
Austria, and \\
Palack\'y University Olomouc \\
Faculty of Science \\
Department of Algebra and Geometry \\
17.\ listopadu 12 \\
771 46 Olomouc \\
Czech Republic \\
helmut.laenger@tuwien.ac.at
\end{document}